\documentclass[seceq]{iis-arXiv}
\usepackage{amsmath, amssymb, amsthm}
\usepackage{mathrsfs}
\usepackage{hyperref}

\usepackage{mathtools}
\mathtoolsset{showonlyrefs=true}

\input cyracc.def


\title{Weil--Petersson curves and Dirichlet finite harmonic functions \\on Riemann surfaces}
\runtitle{Weil--Petersson curves and Dirichlet finite harmonic functions}

\author{%
\name{Katsuhiko \surname{MATSUZAKI}}
\CAE{matsuzak@waseda.jp} 
}

\inst{Department of Mathematics, School of Education, Waseda University, 
\address{Shinjuku, Tokyo 169-8050, Japan}}

\runauthor{K. Matsuzaki}

\support{Research supported by 
Japan Society for the Promotion of Science (KAKENHI 23K25775 and 23K17656)}

\abst{
On two subsurfaces of a Riemann surface divided by a $p$-Weil--Petersson curve $\gamma$,  
we consider the spaces of harmonic functions whose $p$-Dirichlet integrals are finite in the complementary domains of $\gamma$.  
By requiring the coincidence of boundary values on $\gamma$,  
we establish a correspondence between the harmonic functions in these Banach spaces. 
We analyze the operator arising from this correspondence via the composition operator acting on  
the Banach space of $p$-Besov functions on the unit circle.
}

\kword{\kw{integrable Teichm\"uller space}, \kw{Weil--Petersson curve}, \kw{Dirichlet finite harmonic function}, 
\kw{Besov space}, \kw{composition operator}, \kw{Dirichlet principle}}

\begin{document}
\maketitle

\section{Introduction}

Suppose a closed Jordan curve $\gamma$ on a Riemann surface $R$ divides it into two subsurfaces, $R_1$ and $R_2$. 
In this setting, we can study the correspondence between harmonic functions on $R_1$ and $R_2$ that share the same boundary values on $\gamma$, 
provided that $\gamma$ has sufficient regularity. Schippers and Staubach \cite{SS,SS20,SS21} and Shiga \cite{Shiga} investigated this problem, 
focusing on the case where $\gamma$ is a quasicircle and the harmonic functions have finite Dirichlet integrals. When $\gamma$ is a quasicircle, 
the condition of coinciding boundary values on $\gamma$ can be directly formulated using non-tangential limits. 
However, as closely examined in \cite[Theorem 2.4]{SS0}, 
this can be interpreted indirectly via conformal mappings and translate this concept into the correspondence under the composition operator acting on a Hilbert space of boundary functions on the unit circle.
We adapt this definition in this paper.

There are several differences in our approach compared to previous studies.  
First, we assume a higher degree of regularity for $\gamma$.  
A Weil--Petersson curve in the plane is a rectifiable Jordan curve that arises as the image of a specific embedding of the unit circle.  
This embedding can be described in terms of the universal Teich\-m\"ul\-ler space, particularly its subspace associated with integrable Beltrami coefficients (the integrable Teich\-m\"ul\-ler space).  
Bishop \cite{Bi, Bi1} conducted a comprehensive study on Weil--Petersson curves, exploring their connections to various mathematical fields.  
A Weil--Petersson curve can also be defined on a Riemann surface.  
Due to its geometric properties, it is possible to define the boundary function space on $\gamma$,  
which allows us to directly describe the relationship between Dirichlet finite harmonic functions on the complementary domains $R_1$ and $R_2$ of $\gamma$ on a Riemann surface $R$.  
Although this approach is feasible, we opt for a more general framework that includes the case of quasicircles in parallel.  
This approach is only used for presenting a recent result by Wei and Zinsmeister \cite{WZ,WZp}
concerning the characterization of chord-arc curves by the extension of
Dirichlet finite harmonic functions.

Second, we generalize the exponent of the Dirichlet integral from $2$ to an arbitrary $p>1$.  
This generalization aligns with the extension of the integrability exponent of Beltrami coefficients in the integrable Teich\-m\"ul\-ler space.  
The classical theory primarily considers the case $p=2$, where a $2$-Weil--Petersson curve is associated with the $2$-integrable Teich\-m\"ul\-ler space $T_2$.  
By introducing the $p$-integrable Teich\-m\"ul\-ler space $T_p$ for $p>1$,  
we define the notion of a $p$-Weil--Petersson curve, as discussed in \cite{WM-4}.  
For such a curve $\gamma$, it is natural to consider $p$-Dirichlet finite harmonic functions on $R_1$ and $R_2$.  
In this framework, the space of boundary functions on $\gamma$ corresponds to the homogeneous $p$-Besov space $B_p(\mathbb{S})$ on the circle,  
which is well studied in real analysis and offers additional scale parameters $B^s_{p,q}$.  
Let ${\rm HD}_p(R_i)$ denote the space of $p$-Dirichlet finite harmonic functions on $R_i$ $(i=1,2)$.  
After establishing the isomorphism between ${\rm HD}_p(R_i)$ and $B_p(\mathbb{S})$ as complex Banach spaces,  
the transmission operator $\Theta_\gamma: {\rm HD}_p(R_1) \to {\rm HD}_p(R_2)$,  
which encodes the correspondence of harmonic functions with the same boundary values,  
can be realized as an automorphism of $B_p(\mathbb{S})$.

Although we assume that $\gamma$ is a Weil--Petersson curve, most of the results in this paper remain valid even when $\gamma$ is merely a quasicircle,  
as in previous studies.  
In this sense, our results extend some of the known theorems to the setting of $p$-Dirichlet finite harmonic functions,  
by incorporating the $p$-Besov space of boundary functions.  
An exception arises in the final subsection,  
where we analyze the variation of transmission operators under deformations of $R$  
via certain quasiconformal homeomorphisms associated with the $p$-integrable Teich\-m\"ul\-ler space.  
Suppose a sequence of pairs $(R_n,\gamma_n)$ converges to $(R_0,\gamma_0)$ in this Teich\-m\"ul\-ler space.  
We represent the corresponding sequence of transmission operators  
$\Theta_{\gamma_n}: {\rm HD}_p((R_1)_n) \to {\rm HD}_p((R_2)_n)$  
as a sequence of Banach space automorphisms $C(R_n,\gamma_n)$ of $B_p(\mathbb{S})$.  
We then prove that  
$C(R_n,\gamma_n)$ converges strongly to $C(R_0,\gamma_0)$ as $n \to \infty$.


\section{Planar and square integrable case}\label{2}

In this section, we review our problems in the simplest setting of 
Dirichlet finite harmonic functions on planar domains, and motivate us to further generalization.

\subsection{The space of Dirichlet finite harmonic functions}
The space of complex-valued Dirichlet finite harmonic functions $\Phi$ on a planar domain $\Omega$ 
in the extended complex plane $\widehat{\mathbb C}$ is given as
\begin{align*}
{\rm HD}_2(\Omega)=\{\Phi:{\rm harmonic}\ {\rm on}\ \Omega \mid \Vert \Phi \Vert_{D_2}<\infty\}, \quad
\Vert \Phi \Vert_{D_2}=\left(\int_{\Omega}(|\Phi_x|^2+|\Phi_y|^2)dxdy\right)^{1/2}.
\end{align*}
By ignoring the difference in constant functions, we regard ${\rm HD}_2(\Omega)$ as
a complex Hilbert space.

Let $\mathbb D=\{|z|<1\}$ be the unit disk and $\mathbb D^*=\{|z|>1\} \cup\{\infty\}$ its exterior.
We consider {the correspondence between ${\rm HD}_2(\mathbb D)$ and ${\rm HD}_2(\mathbb D^*)$} across the unit circle
$\mathbb S=\{|z|=1\}$. 
The conformal reflection $\lambda(z)=z^*:=(\bar z)^{-1}$ gives directly the correspondence
$$
{\rm HD}_2(\mathbb D) \ni \Phi \longleftrightarrow \Phi^*=\Phi \circ \lambda \in {\rm HD}_2(\mathbb D^*).
$$

In addition, through the boundary values on $\mathbb S$, we can also represent this correspondence.
The space of boundary functions on $\mathbb S$ is the (homogeneous) Besov space or the square integrable Sobolev space of 
half-order differentiability:
\begin{align*}
H^{1/2}(\mathbb S)={B_2(\mathbb S)}=\{\phi \in L_2(\mathbb S) \mid \Vert \phi \Vert_{B_2}<\infty\},\quad
\Vert \phi \Vert_{B_2}=\left(\int_{\mathbb S} \int_{\mathbb S} \frac{|\phi(x)-\phi(y)|^2}{|x-y|^2}dxdy \right)^{1/2}.
\end{align*}
This is also regarded as a complex Hilbert space by ignoring the difference in constant functions.

Every function $\Phi \in {\rm HD}_2(\mathbb D)$ has non-tangential limits almost everywhere on $\mathbb S$,
and this boundary function, denoted by $E(\Phi)$, belongs to $B_2(\mathbb S)$. 
In fact, non-tangential limits exist except on a set of logarithmic capacity zero by the Beurling theorem
(see \cite[Theorem 2.1]{SS0} for an exposition on this fact).
Conversely,
for every $\phi \in B_2(\mathbb S)$, its Poisson integral $P(\phi)$ given by
$$
P(\phi)(z)=\frac{1}{2\pi}\int_{\mathbb S}\frac{1-|\zeta|^2}{|\zeta-z|^2}\phi(\zeta)|d\zeta| 
\qquad (z \in \mathbb D)
$$
belongs to ${\rm HD}_2(\mathbb D)$. Then, the Douglas formula (see \cite[Theorem 2-5]{Ah}) implies the following.

\begin{proposition}\label{Douglas}
The boundary extension operator $E: {\rm HD}_2(\mathbb D) \to B_2(\mathbb S)$ and the Poisson integral operator
$P:B_2(\mathbb S) \to {\rm HD}_2(\mathbb D)$ are isometric linear isomorphisms up to a constant multiple with $P^{-1}=E$. 
\end{proposition}

\subsection{The correspondence across a quasicircle}
We consider the spaces of Dirichlet finite harmonic functions on the complementary domains
divided by a quasicircle.
A closed Jordan curve $\Gamma$ in the complex plane $\mathbb C$ is called a {\em quasicircle} if
it is the image of $\mathbb S$ under some quasiconformal self-homeomorphism $G:\mathbb C \to \mathbb C$. 
If there exists such $G$ with the maximal dilatation not greater than $K \geq 1$, then $\Gamma$ is called
$K$-quasicircle.

\begin{proposition}\label{quasicircle}
A closed Jordan curve $\Gamma$ in $\mathbb C$ is a quasicircle if and only if there exists a constant $C \geq 1$
such that ${\rm diam}(\widetilde{z_1z_2}) \leq C|z_1-z_2|$ for any $z_1, z_2 \in \Gamma$, where
$\widetilde{z_1z_2}$ is the smaller arc in $\Gamma$ between $z_1$ and $z_2$.
\end{proposition}
\smallskip

Let $\Omega_1$ and $\Omega_2$ be simply connected domains in $\widehat{\mathbb C}$
divided by a closed Jordan curve $\Gamma$.
We consider the correspondence between ${\rm HD}_2(\Omega_1)$ and ${\rm HD}_2(\Omega_2)$.
To this end, generalizing the conformal reflection with respect to $\mathbb S$,
we define quasiconformal reflection.
For a quasicircle $\Gamma$,
an anti-quasiconformal self-homeomorphism $\lambda:\widehat{\mathbb C} \to \widehat{\mathbb C}$ is
called a {\em quasiconformal reflection} with respect to $\Gamma$ if $\lambda \circ \lambda={\rm id}$ and
$\lambda|_\Gamma={\rm id}|_\Gamma$.
Any $K$-quasicircle admits a $K^2$-quasiconformal reflection.

For $\Phi_1 \in {\rm HD}_2(\Omega_1)$, $\Phi_1 \circ \lambda$ is a Dirichlet finite Tonelli function on $\Omega_2$.
A Tonelli function is a continuous function with ACL property whose partial derivatives are
locally square integrable (see \cite[p.147]{SN}).
If $\lambda$ is a $K$-quasiconformal reflection, then
\begin{equation}\label{Kqc}
\Vert \Phi_1 \circ \lambda \Vert_{D_2} \leq K^{1/2} \Vert \Phi_1 \Vert_{D_2}.
\end{equation}
{The Dirichlet principle} implies that there exists a unique minimizer $\Phi_2$ of 
$\Vert \cdot \Vert_{D_2}$, which is harmonic,
among all Tonelli functions having the same boundary value as $\Phi_1 \circ \lambda$.
We note that the boundary value of $\Phi_1 \circ \lambda$ is defined on $\Gamma$ 
almost everywhere with respect to the harmonic measure of $\Omega_2$ as explained in the next subsection.
This gives a bounded linear operator 
$\Theta_\Gamma:{\rm HD}_2(\Omega_1) \to {\rm HD}_2(\Omega_2)$ defined by $\Phi_1 \mapsto \Phi_2$.
By considering the converse correspondence ${\rm HD}_2(\Omega_2) \to {\rm HD}_2(\Omega_1)$, we see the following.

\begin{proposition}
For any $K$-quasicircle $\Gamma$, the operator
$\Theta_\Gamma:{\rm HD}_2(\Omega_1) \to {\rm HD}_2(\Omega_2)$ is a Hilbert isomorphism with 
$\Vert \Theta_\Gamma \Vert \leq K$.
\end{proposition}

\subsection{The composition operator}
We examine the correspondence between ${\rm HD}_2(\Omega_1)$ and ${\rm HD}_2(\Omega_2)$
through boundary functions on the quasicircle $\Gamma$. In fact, the extension to $\Gamma$
of the functions in ${\rm HD}_2(\Omega_1)$ and ${\rm HD}_2(\Omega_2)$
can be
understood as the functions in ${\rm HD}_2(\mathbb D)$ and ${\rm HD}_2(\mathbb D^*)$ to $\mathbb S$
by using the normalized Riemann mappings $F_1:\mathbb D \to \Omega_1$ and $F_2:\mathbb D^* \to \Omega_2$. 
The normalization is given by fixing three boundary points $1$, $i$, and $-i$. Hence, $\Gamma$ is also assumed to
pass through these points.
Moreover, in this situation,
we take the normalized quasisymmetric homeomorphism $h_\Gamma:\mathbb S \to \mathbb S$ as $F_1^{-1} \circ F_2|_{\mathbb S}$.
This is uniquely determined 
by a given $K$-quasicircle $\Gamma$ so that the quasisymmetry constant of $h$ depends only on $K$. 
Even if $\Gamma$ is merely a closed Jordan curve not necessarily a quasicircle, we have that $h_\Gamma$ is a homeomorphism.
Conversely, for any normalized quasisymmetric homeomorphism $h$ of $\mathbb S$, 
the conformal welding yields a quasicircle $\Gamma$ for which $h=h_\Gamma$.

Now, applying these operations,
we have the following sequence composing $\Theta_\Gamma$ from ${\rm HD}_2(\Omega_1)$ to ${\rm HD}_2(\Omega_2)$:
\begin{align}\label{compose}
\Theta_\Gamma:\ &{\rm HD}_2(\Omega_1) \overset{(F_1)^*}{\longrightarrow} {\rm HD}_2(\mathbb D) \overset{E_1}{\longrightarrow} B_2(\mathbb S)
\overset{h_\Gamma^*}{\longrightarrow} B_2(\mathbb S) \overset{P_2}{\longrightarrow} {\rm HD}_2(\mathbb D^*)
\overset{(F_2)_*}{\longrightarrow} {\rm HD}_2(\Omega_2).
\end{align}
Here, $E_1$ stands for the non-tangential limit, $P_2$ the Poisson integral, and $\ast$ indicates the pull-back or the push-forward
by the prescribed mappings. The fact that $h_\Gamma^*$ preserves $B_2(\mathbb S)$ is discussed below.

\begin{definition}
For a quasicircle $\Gamma$ and its complementary domains $\Omega_1$ and $\Omega_2$ in $\widehat{\mathbb C}$,
the operator $\Theta_\Gamma:{\rm HD}_2(\Omega_1) \to {\rm HD}_2(\Omega_2)$ defined by \eqref{compose} is
called the {\em transmission operator} across $\Gamma$.
\end{definition}
\smallskip

For an orientation-preserving homeomorphism $h:\mathbb S \to \mathbb S$ in general,
$C_h(\phi)=h^*(\phi)=\phi \circ h$ for a function $\phi$ on $\mathbb S$ is called the {\em composition operator} induced by $h$.
When $\Gamma$ is a quasicircle, the quasisymmetric homeomorphism $h_\Gamma$ 
preserves sets of logarithmic capacity zero though $h_\Gamma$ is not necessarily absolutely continuous.
Since the boundary values of $\Phi \in {\rm HD}_2(\mathbb D)$ are defined by non-tangential limits except on a set of
logarithmic capacity zero, the coincidence of the boundary extensions of harmonic functions in ${\rm HD}_2(\Omega_1)$ and 
${\rm HD}_2(\Omega_2)$
can be formulated by the correspondence under the composition operator $C_{h_\Gamma}$. See \cite[Theorem 2.4]{SS0}.

\smallskip
\begin{remark}
In the preceding work \cite{SS0, SS, SS20, SS21, SS1} and \cite{Shiga}, the transmission is defined by
the coincidence of the non-tangential limits on $\Gamma$ in the above sense for the functions in 
${\rm HD}_2(\Omega_1)$ and ${\rm HD}_2(\Omega_2)$. However, we define this by the correspondence under
the composition operator defined by $\Gamma$. As mentioned above, our definition is compatible with theirs.
In the sequel, we use an expression ``the same boundary values
on $\Gamma$''
for the functions in 
${\rm HD}_2(\Omega_1)$ and ${\rm HD}_2(\Omega_2)$ or for those in more generalized spaces
meaning that they correspond under the composition operator on $\mathbb S$.
In particular, this convention makes it possible that
for a $p$-Dirichlet finite harmonic function $\Phi$ discussed in the next section and 
for a closed Jordan curve $\Gamma$ in general, our arguments work under
the existence of non-tangential limits only almost everywhere. 
\end{remark}
\smallskip

As a special case of the Vodop\cprime yanov theorem for composition operators
(the more general case is presented later), we have the following theorem. See \cite[Theorem 3.1]{NS}.
The if-part follows from the Dirichlet principle in this special case. 

\begin{theorem}\label{compo2}
The composition operator $C_h$ on $B_2(\mathbb S)$
is a Hilbert automorphism of $B_2(\mathbb S)$ if and only if $h$ is quasisymmetric.
\end{theorem}

\begin{corollary}
If $\Gamma$ is a quasicircle, then $\Theta_\Gamma:{\rm HD}_2(\Omega_1) \to {\rm HD}_2(\Omega_2)$ is a Hilbert isomorphism.
Conversely, if this correspondence gives a Hilbert isomorphism of ${\rm HD}_2(\Omega_1)$ onto ${\rm HD}_2(\Omega_2)$, then
$\Gamma$ is a quasicircle.
\end{corollary}

\smallskip
This corollary is shown in \cite[Theorem 2.14]{SS0}, \cite[Theorem 3.5]{SS1}, \cite[Theorem 3.1]{Shiga}, and \cite[Theorem 1.1]{WZ}
under slightly different definitions and assumptions.

\subsection{Dirichlet principle}
Here, we revisit the Dirichlet principle.
Let $H^1(\Omega)=D_2(\Omega)$ be the (homogeneous) Sobolev space of 
locally integrable functions $\Psi$ on $\Omega$ such that the weak derivatives 
satisfy
$$
\Vert \Psi \Vert_{D_2}=\left(\int_{\Omega}(|\Psi_x|^2+|\Psi_y|^2)dxdy\right)^{1/2}<\infty.
$$

\begin{theorem}
The Sobolev space $D_2(\Omega)$
admits the orthogonal direct sum decomposition
\begin{equation}\label{dirichlet2}
D_2(\Omega)={\rm HD}_2(\Omega) \oplus \overline{C_0^\infty(\Omega)},
\end{equation}
where the closure is taken with respect to the Dirichlet norm $\Vert \cdot \Vert_{D_2}$.
\end{theorem}
\smallskip

Indeed, in the case of $\Omega=\mathbb D$, there exists
a bounded trace operator $E:D_2(\mathbb D) \to B_2(\mathbb S)$ which gives the boundary extension of 
$\Psi \in D_2(\mathbb D)$ to $\mathbb S$. Then,
the Poisson integral operator $P:B_2(\mathbb S) \to {\rm HD}_2(\mathbb D)$ is the isometric inverse
for $E$ as in Proposition \ref{Douglas}. This yields the orthogonal direct sum decomposition
$D_2(\mathbb D)={\rm Im}\,P \oplus {\rm Ker}\,E$, where ${\rm Ker}\,E$ is also represented as
$\overline{C_0^\infty(\mathbb D)}$. Pushing forward the functions on $\mathbb D$ by the
Riemann mapping $\mathbb D \to \Omega$, we obtain \eqref{dirichlet2}.
We may regard this orthogonal decomposition as the Dirichlet principle.

\subsection{Chord-arc curves}
In the above arguments, the functions on $\Gamma$ mediate
between ${\rm HD}_2(\Omega_1)$ and ${\rm HD}_2(\Omega_2)$, but they are translated from those on $\mathbb S$ by
the Riemann mapping and the function space on $\Gamma$ is not essentially utilized.
However, if a better regularity on $\Gamma$ are assumed, a certain function space on $\Gamma$ exists
isomorphically between ${\rm HD}_2(\Omega_1)$ and ${\rm HD}_2(\Omega_2)$.
 
\begin{definition}
A closed Jordan curve $\Gamma$ in the complex plane $\mathbb C$ is called a {\em chord-arc curve} if
it is the image of $\mathbb S$ under a bi-Lipschitz self-homeomorphism $G:\mathbb C \to \mathbb C$. 
\end{definition}
\smallskip

A bi-Lipschitz homeomorphism is quasiconformal. Hence, a chord-arc curve is a quasicircle.
Moreover, a chord-arc curve is rectifiable. Arc-length is denoted by $\ell$.

\begin{proposition}
A rectifiable closed Jordan curve $\Gamma$ is a chord-arc curve if and only if there exists a constant $\kappa \geq 1$
such that $\ell(\widetilde{z_1z_2}) \leq \kappa |z_1-z_2|$ for any $z_1, z_2 \in \Gamma$, where
$\widetilde{z_1z_2}$ is the smaller subarc in $\Gamma$ between $z_1$ and $z_2$.
\end{proposition}
\smallskip

By comparing this characterization with Proposition \ref{quasicircle}, we also see that a chord-arc curve is a quasicircle.

Since a chord-arc curve $\Gamma$ is rectifiable, we can directly define the Besov space on $\Gamma$:
\begin{align*}
B_2(\Gamma)=\{\phi \in L_2(\Gamma) \mid \Vert \phi \Vert_{B_2}<\infty\},\quad
\Vert \phi \Vert_{B_2}=\left(\int_{\Gamma} \int_{\Gamma} \frac{|\phi(x)-\phi(y)|^2}{|x-y|^2}dxdy \right)^{1/2},
\end{align*}
where $x$ and $y$ are arc-length parameters of $\Gamma$.
Then, we have the correspondence under the boundary extension by
the non-tangential limits and the integral by the harmonic measures likewise in the case of the unit circle:
$$
{\rm HD}_2(\Omega_1) \overset{E_1}{\underset{P_1}{\rightleftarrows}} B_2(\Gamma)
\overset{E_2}{\underset{P_2}{\leftrightarrows}} {\rm HD}_2(\Omega_2).
$$
We note that by the Lavrentiev theorem for chord-arc curves, the harmonic measures on $\Gamma$ with respect to $\Omega_1$, $\Omega_2$ and 
the arc-length measure on $\Gamma$ are
absolutely continuous with $A_\infty$ derivative (strongly quasisymmetric) to each other. See \cite[p.222]{JK}.

The boundedness of $E$ and $P$ are shown by comparing the norm $\Vert \phi \Vert_{B_2}$ with
$\Vert F^*\phi \Vert_{B_2}$ under the Riemann mapping $F:\mathbb D \to \Omega$ with the use of
Theorem \ref{compo2},
and then applying Proposition \ref{Douglas}.
In addition to this,
Wei and Zinsmeister \cite[Theorem 3]{WZ} also prove the converse statement, which 
gives a characterization of chord-arc curve in terms of the boundedness of
the boundary extension operator.

\begin{theorem}\label{WZ}
For a rectifiable closed Jordan curve $\Gamma$, the boundary extension operator $E$ on ${\rm HD}_2(\Omega)$
defines a Hilbert isomorphism onto $B_2(\Gamma)$ if and only if $\Gamma$ is a chord-arc curve.
\end{theorem}

\section{$p$-Dirichlet finite harmonic functions}\label{3}

Harmonic functions on the unit disk with finite Dirichlet integral are well-studied. 
In this section, we generalize the exponent of the integrability of the derivative from $p=2$ to any $p>1$, 
and lay the foundation for this space.

\subsection{$p$-Dirichlet integral and $p$-Besov functions}

Let $\Phi$ be a complex-valued harmonic function on a planar domain $\Omega$.
For $p > 1$, let
$$
\Vert \Phi \Vert_{D_p}=\left(\int_{\Omega} (\vert \nabla \Phi(z) \vert
\rho_{\Omega}^{-1}(z))^p \rho_{\Omega}^2(z)dxdy\right)^{1/p},
$$
where $|\cdot |$ is the Euclidean norm on $\mathbb C^2$ and $\rho_{\Omega}$ is the hyperbolic density on $\Omega$.
The set of all those {\em $p$-Dirichlet finite} harmonic functions $\Phi$ with $\Vert \Phi \Vert_{D_p}<\infty$ is denoted by 
${\rm HD}_p(\Omega)$.
By ignoring the difference in constant functions, ${\rm HD}_{p}(\Omega)$ is a complex Banach space with norm $\Vert \Phi \Vert_{D_p}$.
For $p=2$, this is the space of Dirichlet finite harmonic functions on $\Omega$ defined in Section \ref{2}.
If $p<p'$, then the inclusion ${\rm HD}_p(\Omega) \hookrightarrow {\rm HD}_{p'}(\Omega)$ is continuous.

We consider the correspondence between ${\rm HD}_p(\mathbb D)$ and ${\rm HD}_p(\mathbb D^*)$ as before.
The conformal reflection $\lambda$ with respect to $\mathbb S$ induces the isometric isomorphism between them.
The space of their boundary functions is given as follows.

\begin{definition}
For an integrable function $\phi$ on the unit circle $\mathbb S$ and for $p > 1$, let
$$
\Vert \phi \Vert_{B_p}=\left(\int_{\mathbb S}\int_{\mathbb S}\frac{|\phi(x)-\phi(y)|^p}{|x-y|^2}dydy\right)^{1/p},
$$
and the set of all those $\phi$ with $\Vert \phi \Vert_{B_p}<\infty$ is denoted by $B_p(\mathbb S)$ and called 
the (homogeneous critical) {\it $p$-Besov space}. This is regarded as a complex Banach space.
\end{definition}
\smallskip

Every harmonic function $\Phi$ in ${\rm HD}_p(\mathbb D)$ has non-tangential limits almost everywhere on $\mathbb S$
because ${\rm HD}_p(\mathbb D)$ is contained in the harmonic Hardy space $h_1$. 
This boundary function is denoted by $E(\Phi)$. Then, $E(\Phi)$ belongs to $B_p(\mathbb S)$.
Conversely, for any $\phi \in B_p(\mathbb S)$, its Poisson integral $P(\phi)$
belongs to ${\rm HD}_p(\mathbb D)$. Moreover, the following property is known. See \cite[Chap.V, Proposition 7$'$]{St}.

\begin{theorem}\label{PE}
For $p>1$, the boundary extension operator $E: {\rm HD}_p(\mathbb D) \to B_p(\mathbb S)$ and the Poisson integral operator
$P:B_p(\mathbb S) \to {\rm HD}_p(\mathbb D)$ are Banach isomorphisms with $P^{-1}=E$. 
\end{theorem}
\smallskip

Thus, we have the correspondence ${\rm HD}_p(\mathbb D) \cong B_p(\mathbb S) \cong {\rm HD}_p(\mathbb D^*)$ 
under the Banach isomorphisms. We generalize this to the function spaces on
quasidisks $\Omega_1$ and $\Omega_2$ divided by a quasicircle $\Gamma$.

\subsection{The generalized transmission operator}

For the composition operator $C_h$ on $B_p(\mathbb S)$ given by 
an orientation-preserving homeomorphism  $h:\mathbb S \to \mathbb S$,
Theorem \ref{compo2} can be generalized to the Vodop\cprime yanov theorem as follows.

\begin{theorem}\label{vod}
The composition operator $C_h$ on $B_p(\mathbb S)$ for $p>1$
defines a Banach automorphism of $B_p(\mathbb S)$ if and only if $h$ is quasisymmetric.
\end{theorem}
\smallskip

The proof is more complicated than that of Theorem \ref{compo2}.
We use the trace operator $E:B^{2/p}_{p,p}(\mathbb D) \to B_p(\mathbb S)$ and the real interpolation
$B^{2/p}_{p,p}(\mathbb D)=({\rm BMO}(\mathbb D), W^1_2(\mathbb D))_{2/p,p}$ for $p>2$. Then, the composition operator induced by 
a quasiconformal self-homeomorphism of $\mathbb D$ is bounded on ${\rm BMO}(\mathbb D)$ and $W^1_2(\mathbb D)$.
See \cite[Theorem 1.3]{BS} and \cite[Theorem 12]{B2}.

\smallskip
\begin{remark}
Several statements and definitions in this paper can be also generalized to the case of $p =1$,
but Theorem \ref{vod} is critical for our arguments in the sense that the condition $p>1$ is required.
The treatment for the case $p=1$ is developed in \cite{Mnew}.
\end{remark}
\smallskip

Now, we consider the correspondence between ${\rm HD}_p(\Omega_1)$ and ${\rm HD}_p(\Omega_2)$
for a quasicircle $\Gamma$. As before,
by the normalized Riemann mappings $F_1:\mathbb D \to \Omega_1$ and $F_2:\mathbb D^* \to \Omega_2$, and
the orientation-preserving homeomorphism 
$h_\Gamma=F_1^{-1} \circ F_2|_{\mathbb S}$, we obtain the sequence for the definition of the transmission operator
\begin{align}\label{sequence}
\Theta_\Gamma:\ &{\rm HD}_p(\Omega_1) \overset{(F_1)^*}{\longrightarrow} {\rm HD}_p(\mathbb D) \overset{E_1}{\longrightarrow} B_p(\mathbb S)
\overset{h_\Gamma^*}{\longrightarrow} B_p(\mathbb S) \overset{P_2}{\longrightarrow} {\rm HD}_p(\mathbb D^*)
\overset{(F_2)_*}{\longrightarrow} {\rm HD}_p(\Omega_2).
\end{align}
Then, Theorems \ref{PE} and \ref{vod} imply the following.

\begin{corollary}\label{main1}
If $\Gamma$ is a quasicircle, then $\Theta_\Gamma:{\rm HD}_p(\Omega_1) \to {\rm HD}_p(\Omega_2)$ is a Banach isomorphism.
Conversely, if this correspondence defines a Banach isomorphism of ${\rm HD}_p(\Omega_1)$ onto ${\rm HD}_p(\Omega_2)$, then
$\Gamma$ is a quasicircle.
\end{corollary}
\smallskip

When $\Gamma$ is a chord-arc curve, the $p$-Besov space $B_p(\Gamma)$ on $\Gamma$ for $p>1$ is defined by the same way as
in the case of $p=2$:
\begin{align*}
B_p(\Gamma)=\{\phi \in L_p(\Gamma) \mid \Vert \phi \Vert_{B_p}<\infty\},\quad
\Vert \phi \Vert_{B_p}=\left(\int_{\Gamma} \int_{\Gamma} \frac{|\phi(x)-\phi(y)|^p}{|x-y|^2}dxdy \right)^{1/p}.
\end{align*}
Since $\Vert \phi \Vert_{B_p}$ for $\phi \in B_p(\Gamma)$ is comparable to
$\Vert F^*\phi\Vert_{B_p}$ by Theorem \ref{vod} as before, 
Theorem \ref{PE} implies the generalization of the if-part of Theorem \ref{WZ}.
See \cite[Theorem 3]{WZp}.

\begin{corollary}
If $\Gamma$ is a chord-arc curve, the boundary extension operator 
$E:{\rm HD}_p(\Omega) \to B_p(\Gamma)$ is a Banach isomorphism for $p>1$.
\end{corollary}
\smallskip

For $p \geq 2$, the generalization of Theorem \ref{WZ} to this case is given in \cite[Theorem 1]{WZp}.

\subsection{$p$-Dirichlet principle}
For $p>1$, let $D_p(\Omega)$ be the (homogeneous) Sobolev space of locally integrable and weakly differentiable functions $\Phi$ on 
a planar domain $\Omega$
with $\Vert \Phi \Vert_{D_p}<\infty$, defined analogously to the case of $p=2$. 
Together with the Banach isomorphism $P:B_p(\mathbb S) \to {\rm HD}_p(\mathbb D)$,
the boundedness of the trace operator $E:D_p(\mathbb D) \to B_p(\mathbb S)$ implies
$D_p(\mathbb D)={\rm Im}\,P \oplus {\rm Ker}\,E$, and thus yields
the following $p$-Dirichlet principle, generalizing the case of $p=2$.
Through the Riemann mapping, we can assert this for any Jordan domain $\Omega \subset \widehat{\mathbb C}$.

\begin{theorem}\label{dirichletp}
For any $p>1$, the Sobolev space
$D_p(\Omega)$ admits the topological direct sum decomposition
$$
D_p(\Omega)={\rm HD}_p(\Omega) \oplus \overline{C_0^\infty(\Omega)}.
$$
\end{theorem}
\smallskip

The existence and the boundedness of the above trace operator $E$ is known as the Uspenski\u{\i} theorem.
See \cite[Theorem 1.1]{MR} and \cite[Section 10.1.1, Theorem 1]{Maz}.

We note that in contrast to the case $p=2$, $p$-Dirichlet principle does not work for the argument using
the quasiconformal reflection $\lambda$ with respect to a quasicircle $\Gamma$. This is because
we cannot estimate the $p$-Dirichlet norm of $\Phi \circ \lambda$ for $\Phi \in {\rm HD}_p(\Omega)$.

\section{The $p$-integrable Teich\-m\"ul\-ler space and $p$-Weil--Petersson curves}

In this section, we introduce the integrable Teich\-m\"ul\-ler space as the parameter space for Weil--Petersson curves in the plane. 

\subsection{$p$-integrable Teich\-m\"ul\-ler space}

For $p>1$, the $p$-integrable Teich\-m\"ul\-ler space $T_p$
is defined in the universal Teich\-m\"ul\-ler space $T$. 
Basic definitions and fundamental results on the universal Teich\-m\"ul\-ler space can be found in 
a textbook \cite{Le}.

\begin{definition}
For $p > 1$, a Beltrami coefficient $\mu \in M(\mathbb D)$, a measurable function on 
the unit disk $\mathbb D$ with $\Vert \mu \Vert_\infty<1$,
is $p$-integrable 
with respect to the hyperbolic metric $\rho_{\mathbb D}(z)|dz|$ if it satisfies
$$
\Vert \mu \Vert_p=\left(\int_{\mathbb D} |\mu(z)|^p \rho^2_{\mathbb D}(z)dxdy\right)^{1/p}<\infty.
$$
The set of all such $p$-integrable Beltrami coefficients is denoted by $M_p(\mathbb D)$.
The topology of $M_p(\mathbb D)$ is defined by the norm $\Vert \mu \Vert_p+\Vert \mu \Vert_\infty$.
The space $M_p(\mathbb D^*)$ is defined similarly on $\mathbb D^*$.
\end{definition}
\smallskip

For $p > 1$, the {\it $p$-integrable Teich\-m\"ul\-ler space} $T_p$ is the
set of all normalized (fixing $1$, $i$, and $-i$)
quasisymmetric homeomorphisms $h:\mathbb S \to \mathbb S$
that can be extended to quasiconformal self-homeomorphisms of $\mathbb D$
whose complex dilatations belong to 
$M_p(\mathbb D)$. 

The measurable Riemann mapping theorem asserts that
for any $\mu \in M(\mathbb D)$, there exists a unique normalized quasiconformal
self-homeomorphism $H^\mu$ of $\mathbb D$ whose complex dilatation is $\mu$.
By this correspondence from $\mu$
to the quasisymmetric homeomorphisms $h^\mu=H^\mu|_{\mathbb S}$, we have
the {\em Teich\-m\"ul\-ler projection} $\pi: M(\mathbb D) \to T$.
An element $h^\mu$ of $T_p$ can be represented by 
the Teich\-m\"ul\-ler equivalence class $[\mu]$ for $\mu \in M_p(\mathbb D)$.
We provide the quotient topology induced from $M_p(\mathbb D)$ for $T_p$.

The integrable Teich\-m\"ul\-ler space $T_p$ possesses the group structure as a subgroup of 
the universal Teich\-m\"ul\-ler space $T$,
where $T$ is regarded as the group of all normalized quasisymmetric homeomorphisms of $\mathbb S$.
For $\mu_1, \mu_2 \in M(\mathbb D)$ with $\pi(\mu_i)=h_i$ $(i=1,2)$, the group operation in $T$ 
is defined and denoted by
$[\mu_1] \ast [\mu_2]=\pi(\mu_1 \ast \mu_2)=h_1 \circ h_2$.

The following claim is in \cite[Chap.1, Theorem 3.8]{TT} for $p=2$ and \cite[Theorem 6.1]{WM-4} in general.

\begin{proposition}\label{group}
$T_p$ is a topological group. 
\end{proposition}

\subsection{Analytic Besov space}

Let $G_\mu$ denote the  
quasiconformal self-homeomorphism of 
$\widehat{\mathbb C}$ satisfying $G_\mu(\infty)=\infty$, $G_\mu(0)=0$, and $(G_\mu)'(0)=1$
whose complex dilatation
is $\mu$ on $\mathbb D^*$ and $0$ on $\mathbb D$.
This normalization is applied only for the following definition.
The pre-Schwarzian derivative map $L:M_p(\mathbb D^*) \to {\mathcal B}_p(\mathbb D)$
is defined by the
correspondence $\mu \mapsto \log (G_\mu|_{\mathbb D})'$, where 
${\mathcal B}_p(\mathbb D)$ is the Banach space of holomorphic functions $\Phi$ on $\mathbb D$ with
$$
\Vert \Phi \Vert_{{\mathcal B}_p}=\left(\int_{\mathbb D} (|\Phi'(z)| \rho^{-1}_{\mathbb D}(z))^p
\rho_{\mathbb D}^2(z)dxdy \right)^{1/p}<\infty.
$$
This is called
the {\it $p$-analytic Besov space}. See \cite[Section 5]{Zhu}.
By ignoring the difference in constant functions, ${\mathcal B}_p(\mathbb D)$ 
is a complex Banach space with norm $\Vert \Phi \Vert_{\mathcal B_p}$. 

The following claim is proved in \cite[Theorem 2.5]{Sh} for $p=2$, \cite[Theorem 2.5]{TS} for $p \geq 2$,
and \cite[Section 3]{WM-4} for general $p$.

\begin{proposition}\label{pre-Bers}
The pre-Schwarzian derivative map $L:M_p(\mathbb D^*) \to {\mathcal B}_p(\mathbb D)$ is
holomorphic with local holomorphic right inverse at every point in the image $L(M_p(\mathbb D^*))$.
\end{proposition}
\smallskip

Let $J(\Phi)=\Phi''-(\Phi')^2/2$. Then, $J(\Phi)$ for 
$\Phi=\log (G_\mu|_{\mathbb D})' \in {\mathcal B}_p(\mathbb D)$ is the Schwarzian derivative 
of $G_\mu|_{\mathbb D}$, and it belongs to
the Banach space ${\mathcal A}_p(\mathbb D)$ of holomorphic functions $\Psi$
on $\mathbb D$ with
$$
\Vert \Psi \Vert_{{\mathcal A}_p}=\left(\int_{\mathbb D} (|\Psi(z)| \rho^{-2}_{\mathbb D}(z))^p
\rho_{\mathbb D}^2(z)dxdy \right)^{1/p}<\infty.
$$
Moreover, $J\circ L$ is factored through the Teich\-m\"ul\-ler projection $\pi:M_p(\mathbb D^*) \to T_p$ to induce
a well-defined injection $\alpha:T_p \to {\mathcal A}_p(\mathbb D)$ such that
$\alpha \circ \pi=J \circ L$. This map $\alpha$ is called the {\it Bers embedding}.

\begin{proposition}\label{Bers}
The Bers embedding $\alpha$ is
a homeomorphism onto a contractible domain $\alpha(T_p)=J \circ L(M_p(\mathbb D^*))$ in 
${\mathcal A}_p(\mathbb D)$.
\end{proposition}
\smallskip

This provides the complex Banach structure for $T_p$. See \cite{Cu, Gu, Sh, TT, TS} for pioneering work
and \cite[Theorem 4.1]{WM-1} for the general case.

We mention the relationship among ${\mathcal B}_p(\mathbb D)$, ${\rm HD}_p(\mathbb D)$, and $B_p(\mathbb S)$.
For $\phi \in B_p(\mathbb S)$, the {\it Hilbert transform} is defined by the singular integral
$$
{\mathcal H}(\phi)(x)=\frac{1}{\pi i}\,{\rm p.v} \int_{\mathbb S} \frac{\phi(\zeta)}{\zeta-x}d \zeta \quad (x \in \mathbb S).
$$
Then, ${\mathcal H}(\phi)$ belongs to $B_p(\mathbb S)$.
Moreover, the {\it Szeg\"o projection} is defined by the Cauchy integral
$$
{\mathcal S}(\phi)(z)=\frac{1}{2\pi i} \int_{\mathbb S} \frac{\phi(\zeta)}{\zeta-z}d \zeta \quad (z \in \mathbb D).
$$
Then, ${\mathcal S}(\phi)$ belongs to the $p$-analytic Besov space $\mathcal B_p(\mathbb D)$.

\begin{proposition}
The following are satisfied:
\begin{itemize}
\item[$(1)$]
${\mathcal H}:B_p(\mathbb S) \to B_p(\mathbb S)$ is a Banach automorphism with ${\mathcal H} \circ {\mathcal H}=I$;
\item[$(2)$]
${\mathcal S}:B_p(\mathbb S) \to \mathcal B_p(\mathbb D)$ is a bounded linear map such that 
$E \circ {\mathcal S}:B_p(\mathbb S) \to B_p(\mathbb S)$ is a bounded projection onto 
$E(\mathcal B_p(\mathbb D))$ satisfying 
$E \circ {\mathcal S}=\frac{1}{2}(I +\mathcal H)$.
\end{itemize}
\end{proposition}
\smallskip

Moreover, defining
$$
{\mathcal S}^*(\phi)(z)=\frac{1}{2\pi i} \int_{\mathbb S} \frac{\phi(\zeta)}{\zeta-z}d \zeta \quad (z \in \mathbb D^*)
$$
where the orientation of the line integral on $\mathbb S$ is taken in the opposite direction for $\mathcal S$, we see that
$E \circ {\mathcal S}^*$ coincides with the bounded projection $\frac{1}{2}(I -\mathcal H)$ 
onto $E(\mathcal B_p(\mathbb D^*))$.
Then, the identification under the boundary extension operator $E$ yields
\begin{equation}\label{identify}
{\rm HD}_p(\mathbb D) \cong B_p(\mathbb S) \cong \mathcal B_p(\mathbb D) \oplus \mathcal B_p(\mathbb D^*).
\end{equation}
For $z \in \mathbb D$, $z^*=\lambda(z) \in \mathbb D^*$ 
denotes the reflection point with respect to $\mathbb S$. For $\Phi \in 
\mathcal B_p(\mathbb D^*)$, let $\Phi^*(z):=\Phi(z^*)$ be 
the anti-holomorphic function defined on $\mathbb D$,
and the set of all such functions $\Phi^*$ coincides with
the complex conjugate $\overline{\mathcal B_p(\mathbb D)}$.
Thus, $\mathcal B_p(\mathbb D^*)$ and $\overline{\mathcal B_p(\mathbb D)}$ are 
identified by the conformal reflection $\lambda$, and the
above identification \eqref{identify} is nothing but the decomposition 
$$
{\rm HD}_p(\mathbb D) \cong \mathcal B_p(\mathbb D) \oplus \overline{\mathcal B_p(\mathbb D)}
$$
of a harmonic function into holomorphic and anti-holomorphic parts.

\subsection{$p$-Weil--Petersson curves}

We define Weil--Petersson curves and parametrize them by
the Teich\-m\"ul\-ler space $T_p$. For any $\mu \in M_p(\mathbb D)$, let $F_\mu$ denote
the quasiconformal self-homeomorphism of $\widehat{\mathbb C}$ whose complex dilatation is $\mu$ on $\mathbb D$ and
$0$ on $\mathbb D^*$
with the normalization
fixing $1$, $i$, and $-i$.

\begin{definition}
For $\mu \in M_p(\mathbb D^*)$, the image $\Gamma$ of $\mathbb S$
under the quasiconformal homeomorphism $F_\mu$ is called a $p$-Weil--Petersson curve. 
\end{definition}
\smallskip

In fact, $F_\mu|_{\mathbb S}=F_\nu|_{\mathbb S}$ if and only if $\pi(\mu)=\pi(\nu)$, and thus
a $p$-Weil--Petersson curve $\Gamma$ is determined by $[\mu] \in T_p$. 

A $p$-Weil--Petersson curve is a chord-arc curve. There are a couple of explanations for this.
One is to show the inclusion $B_p(\mathbb S) \subset {\rm VMO}(\mathbb S)$ (see \cite[Proposition 2.2]{WM-0}).
Another is to show that $|\mu(z)|^2 \rho_{\mathbb D^*}(z)dxdy$ is a Carleson measure on $\mathbb D^*$
for $\mu \in M_p(\mathbb D^*)$.

The following geometric characterization of $2$-Weil--Petersson curve is given by Bishop \cite[Theorem 1.3]{Bi1}.
The generalization to any $p>1$ is a problem.

\begin{theorem}
A rectifiable closed Jordan curve $\Gamma$ is a $2$-Weil--Petersson curve if and only for any $z_0 \in \Gamma$,
the $n$-gon $\Gamma_{n}(z_0)$ made of evenly distributed $n$ points $z_0, z_1, \ldots, z_n=z_0$ on $\Gamma$
with respect to the arc length satisfies
$$
\sum_{n=1}^\infty 2^n\{\ell(\Gamma)-\ell(\Gamma_{n}(z_0))\}<\infty.
$$ 
\end{theorem}

For $\mu_1 \in M_p(\mathbb D)$ and $\mu_2 \in M_p(\mathbb D^*)$,
let $G(\mu_1, \mu_2)$ be the normalized quasiconformal self-homeomorphism of $\widehat{\mathbb C}$
whose complex dilatation is $\mu_1$ on $\mathbb D$ and $\mu_2$ on $\mathbb D^*$. 
Let $\bar \mu_2$ be the complex dilatation in $M_p(\mathbb D)$ given by the reflection of $\mu_2$,
and $\bar \mu_2^{-1}$ the complex dilatation of the normalized quasiconformal self-homeomorphism $(H^{\bar \mu_2})^{-1}$ of $\mathbb D$.
Moreover, $\mu_1 \ast \bar \mu_2^{-1}$ denotes the complex dilatation of $H^{\mu_1} \circ H^{\bar \mu_2^{-1}}$.

\begin{proposition}\label{double}
For $\mu_1 \in M_p(\mathbb D)$ and $\mu_2 \in M_p(\mathbb D^*)$, a quasicircle $\Gamma=G(\mu_1, \mu_2)(\mathbb S)$
is a $p$-Weil--Petersson curve, and it coincides with $F_{\mu_1 \ast \bar \mu_2^{-1}}(\mathbb S)$.
\end{proposition}

\begin{proof}
By $G(\mu_1, \mu_2) \circ H^{\bar \mu_2^{-1}}=F_{\mu_1 \ast \bar \mu_2^{-1}}$, we have the statements. 
\end{proof}

\subsection{Variation of the transmission operators}
We consider the variation of the transmission operator $\Theta_\Gamma$ when $\Gamma$ varies.
For $[\mu] \in T_p$,  
let $\Gamma=F_\mu(\mathbb S)$ be 
a $p$-Weil--Petersson curve. Let $\Omega_1$ and $\Omega_2$ be the complementary domains of 
$\widehat{\mathbb C}$ divided by $\Gamma$, and $F_1:\mathbb D \to \Omega_1$, $F_2:\mathbb D^* \to \Omega_2$
the normalized Riemann mappings. In the sequence in \eqref{sequence} composing the transmission operator $\Theta_\Gamma$,
the quasisymmetric homeomorphism $h_\Gamma:\mathbb S \to \mathbb S$ is given as 
$F_1^{-1} \circ F_2|_{\mathbb S}$.
In the present circumstances, $h_\Gamma^{-1}$ is also identified with $[\mu]=\pi(\mu)$ in $T_p$.
We remark that this inverse does not cause any trouble in later arguments because $T_p$ is a topological group mentioned in Proposition \ref{group}.

We make the operators $\Theta_\Gamma$ act on the same space. Namely, ignoring the isomorphic relations in \eqref{sequence},
the dependence of $\Theta_\Gamma$ is represented by the composition operator $C_{h_\Gamma}$:
$$
{\rm HD}_p(\Omega_1) \cong {\rm HD}_p(\mathbb D) \cong B_p(\mathbb S) \overset{C_{h_\Gamma}}{\longrightarrow} 
B_p(\mathbb S) \cong {\rm HD}_p(\mathbb D^*) \cong {\rm HD}_p(\Omega_2).
$$

Concerning the operator norm of the composition operators, 
the following general claim is known (see \cite[Remark 3.3]{BS}).
We note that to apply this proposition to the transmission operator
$\Theta_\Gamma$, we do not have to assume that $\Gamma$ is a Weil--Petersson curve.
Even to the original setting where $\Gamma$ is a quasicircle,
this is applicable.

\begin{proposition}
The operator norm of $C_h:B_p(\mathbb S) \to B_p(\mathbb S)$ 
can be estimated only by the quasisymmetry constant of $h$ $($or 
the maximal dilatation of a quasiconformal extension of $h$$)$.
\end{proposition}
\smallskip

Moreover, in the case $p=2$, the strong convergence $C_{h_n} \to C_{h}$ as $h_n \to h$ in $T$
is proved by \cite[Proposition 10.1]{Sh}. 

\begin{theorem}\label{strong0}
Suppose that $h_n$ converges to $h$ in $T$. 
Then, the composition operators $C_{h_n}$ acting on $B_2(\mathbb S)$
converge to $C_h$ strongly.
Namely,
$$
\Vert C_{h_n}(\phi)-C_h(\phi) \Vert_{B_2} \to 0 \quad (n \to \infty)
$$
for every $\phi \in B_2(\mathbb S)$.
\end{theorem}
\smallskip

However, in the general case $p>1$, 
we do not know whether this is valid or not.
Instead, the following claim is known under the stronger assumption.

\begin{proposition}\label{strong}
Suppose that $h_n$ converges to $h$ in $T_p$. 
Then, $C_{h_n}$ converges to $C_h$ strongly. 
\end{proposition}

\begin{proof}
This follows from the fact that
the map $B_p(\mathbb S) \times T_p \to B_p(\mathbb S)$ defined by
$(\phi,h) \mapsto C_h(\phi)$ is continuous. See \cite[Theorem 6.3]{WM-4}.
\end{proof}

If $C_{h_n}$ converges to $C_h$ strongly in general, then 
$\liminf_{n \to \infty}\Vert C_{h_n} \Vert \geq \Vert C_h \Vert$ is satisfied.
When $p=2$, if $h_n \to h$ in $T$, then $\lim_{n \to \infty}\Vert C_{h_n} \Vert = \Vert C_h \Vert$.
This is because for $\Phi \in {\rm HD}_2(\mathbb D) \cong B_2(\mathbb S)$ and a $K$-quasiconformal self-homeomorphism $H$ of $\mathbb D$,
we have $\Vert \Phi \circ H \Vert_{D_2} \leq K^{1/2} \Vert \Phi \Vert_{D_2}$.

\section{Generalization to $p$-Dirichlet harmonic functions on Riemann surfaces}\label{5}

We have dealt with the correspondence between the spaces of 
Dirichlet finite harmonic functions on planar domains separated by a quasicircle. 
In this section, we consider this problem for 
a quasicircle on a Riemann surface. We define the transmission operator also
in this setting and represent it using the operators we have defined in the case of planar domains.

\subsection{$p$-Dirichlet finite harmonic functions on a Riemann surface}
We can generalize the definition of Dirichlet finite harmonic function on a planar domain to a Riemann surface.

\begin{definition}
A harmonic function $\Phi$ on a Riemann surface $R$ with hyperbolic metric $\rho_R(z)|dz|$ is
{\em $p$-Dirichlet finite} for $p>1$ if
$$
\Vert \Phi \Vert_{D_p}
=\left(\int_{R}(\vert \nabla \Phi(z)\vert \rho_R^{-1}(z))^p\rho_R^2(z)dxdy\right)^{1/p}<\infty.
$$
\end{definition}

We note that $\vert \nabla \Phi(z)\vert \rho_R^{-1}(z)$ is a function on $R$ defined independently on the local coordinates $z$.
The set of all $p$-Dirichlet finite harmonic functions on $R$ is denoted by {${\rm HD}_p(R)$}.
This is a complex Banach space 
by ignoring the difference in constant functions.
If $p<p'$, then the inclusion ${\rm HD}_p(R) \hookrightarrow {\rm HD}_{p'}(R)$ is continuous.

We consider the Dirichlet principle on a Riemann surface $R$.
In particular, we apply this to a bordered Riemann surface.
For $p>1$, let $D_p(R)$ be the Banach space of weakly differentiable functions 
$\Phi$ on $R$
with $\Vert \Phi \Vert_{D_p}<\infty$.
Then, the $p$-Dirichlet principle can be formulated as
the topological direct sum decomposition
\begin{equation}\label{r-dirichlet}
D_p(R)={\rm HD}_p(R) \oplus \overline{C_0^\infty(R)}.
\end{equation}

In the case $p=2$, the classical method of orthogonal projection for 
Tonelli functions (see e.g. \cite[p.163, Corollary 2]{SN}) also work in this setting of distribution. Thus, \eqref{r-dirichlet}
is valid for $p=2$.
In particular, when $R_0$ is a compact bordered Riemann surface, we can extract
the harmonic part from a function in $D_2(R_0)$ having the same boundary extension to $\partial R_0$.

\begin{proposition}\label{theDirichlet}
Among all functions in $D_2(R_0)$ having the same boundary extension to $\partial R_0$,
there exists a unique function minimizing the Dirichlet norm, which is harmonic.
\end{proposition}
\smallskip

Later in this section, we see that formulation \eqref{r-dirichlet} is valid
for a compact bordered Riemann surface $R_0$ also in the case $p>1$ though the harmonic function
is not necessarily the minimizer of the $p$-Dirichlet norm as stated in Proposition \ref{theDirichlet}.

\subsection{The transmission on a Riemann surface}
Let $\gamma$ be a quasicircle in a Riemann surface $R$ that divides $R$ into
subsurfaces $R_1$ and $R_2$. Here,
a closed Jordan curve $\gamma$ on a Riemann surface $R$ is a quasicircle by definition if
there exists an annular neighborhood $U$ of $\gamma$ and a conformal embedding $\sigma:U \to \mathbb C$
such that $\sigma(\gamma)$ is a quasicircle on $\mathbb C$.
This is well defined independently of the choice of $U$ and $\sigma$ (see \cite[Definition 2.1]{Shiga}).
We regard each $R_i$ $(i=1,2)$ as a bordered Riemann surface and also as
a complete hyperbolic surface having 
$\gamma$ as the boundary at infinity. The hyperbolic metric $\rho_{R_i}(z)|dz|$ is defined in this sense.
For simplicity, we assume that the bordered Riemann surface $R_i$ is compact.
The cases of non-compact surfaces with plural boundary components are also treated in \cite{SS21} and \cite{Shiga}, 
but we do not give such generalization in this paper.
In this situation, we consider the correspondence between
${\rm HD}_p(R_1)$ and ${\rm HD}_p(R_2)$.

The transmission operator $\Theta_\gamma:{\rm HD}_p(R_1) \to {\rm HD}_p(R_2)$ is defined as follows.
This technique of sending $\gamma$ to a plane curve $\Gamma \subset \mathbb C$ and applying the results for $\Gamma$
originates in Schippers and Staubach \cite{SS,SS20,SS21}.
Let $\Phi_1 \in {\rm HD}_p(R_1)$.
We choose an annular neighborhood $U$ of $\gamma$ and 
a conformal embedding $\sigma:U \to \mathbb C$ with $\sigma(\gamma)=\Gamma$.
Let $U_i=U \cap R_i$ and $V_i=\sigma(U_i) \subset \Omega_i$ $(i=1,2)$.
Let $\Psi_1=\sigma_*(\Phi_1|_{U_1})$, which is a harmonic function on $V_1$
whose $p$-Dirichlet norm defined as a function on $\Omega_1$ is finite by the comparison of
hyperbolic densities between $R_1$ and $\Omega_1$.
Then, we extend $\Psi_1$ to $\Omega_1$ as a function in $D_p(\Omega_1)$.
By Theorem \ref{dirichletp}, we obtain $\widetilde \Psi_1 \in {\rm HD}_p(\Omega_1)$
that has ``the same boundary values on $\Gamma$'' as $\Psi_1$.

We apply the transmission operator $\Theta_\Gamma$ to $\widetilde \Psi_1 \in {\rm HD}_p(\Omega_1)$.
By Corollary \ref{main1}, we obtain $\widetilde \Psi_2=\Theta_\Gamma(\widetilde \Psi_1) \in {\rm HD}_p(\Omega_2)$.
Let $\Psi_2$ be the restriction of $\widetilde \Psi_2$ to $V_2$.
We want to choose $\Phi_2 \in {\rm HD}_p(R_2)$ out of $\sigma^*\Psi_2$ so that their 
boundary extensions to $\gamma$ are the same.
We explain how this is possible in the next subsection.
Then, 
the correspondence $\Phi_1 \mapsto \Phi_2$ defines 
the transmission operator $\Theta_\gamma:{\rm HD}_p(R_1) \to {\rm HD}_p(R_2)$.

We note that $\Theta_\gamma$ is constructed from $\Theta_\Gamma$ with the aid of the conformal embedding $\sigma$ of
the neighborhood $U$ of $\gamma$ such that $\Gamma=\sigma(\gamma)$, but this is
defined independently of the choice of $U$ or $\sigma$. 
Moreover, this is defined even for a closed Jordan curve $\gamma$ not necessarily a quasicircle.
This is because on a dense subset of ${\rm HD}_p(R_1)$ consisting of those continuously extendable to $\gamma$,
$\Theta_\gamma$ is determined uniquely.

\subsection{Description of the transmission operator}
We extract the essential step of the arguments for constructing the transmission operator as the following theorem.
For $p=2$, this is given in \cite[Theorem 3.24]{SS21} in a more general case where the boundary consists of multiple curves.
A similar claim for a more general Riemann surface in the case $p=2$
is given in \cite{Shiga} by traditional arguments on open Riemann surfaces.

\begin{theorem}\label{extension}
Let $R_0$ be a compact bordered Riemann surface whose boundary $\partial R_0$ is a simple closed curve $\gamma$.
Let $\sigma$ be a conformal homeomorphism of a neighborhood $U_0$ of $\gamma$ in $R_0$ onto a neighborhood 
$V_0$ of $\mathbb S$ in $\mathbb D$.
Then, the boundary extension map 
$$
E_{\partial R_0}:{\rm HD}_p(R_0) \to B_p(\mathbb S) \cong {\rm HD}_p(\mathbb D)
$$ 
taking the boundary values on $\gamma$ and sending them to $\mathbb S$ by $\sigma$ is bijective. Moreover,
the operator $E_{\partial R_0}$ is bounded, and hence
this is a Banach isomorphism.
\end{theorem}
\smallskip

The proof is owed to the following lemma. In the case $p=2$, this is obtained in \cite[Theorem 3.22]{SS20}.
The harmonic measure on $R_0$ is given as 
the line integral by $-\ast dg(\cdot,z)$ over $\partial R_0$,
where $g(\cdot,z)$ is the Green function on $R_0$ with pole $z$. In the proof of this lemma,
we take the double of $R_0$ with respect to $\partial R_0$ and apply the normal family argument on it.
This idea stems from \cite[Proposition 3.3]{Shiga},
where a similar claim is proved for a general bordered Riemann surface in the case $p=2$.

\begin{lemma}\label{new}
Let $U_0$ be a neighborhood of the boundary $\gamma_0=\partial R_0$ of a compact bordered Riemann surface $R_0$.
Let $\Psi$ be a harmonic function on $U_0$ whose $p$-Dirichlet norm with respect to $\rho_{R_0}$ is finite.
Then, the integral $\Phi=P_{\partial R_0} \circ E_{\partial R_0}(\Psi)$ of the boundary values of $\Psi$ on $\gamma_0$ by the harmonic measure on $R_0$
belongs to ${\rm HD}_p(R_0)$. Moreover, the operator $P_{\partial R_0}:B_p(\mathbb S) \to {\rm HD}_p(R_0)$
induced by this correspondence $\Psi \mapsto \Phi$ is bounded.
\end{lemma}

\begin{proof}
In order to prove that $\Phi \in {\rm HD}_p(R_0)$,
it suffices to show that $\Phi|_{U_0}$ has finite $p$-Dirichlet integral.
To this end, we consider $\Phi_0=\Phi-\Psi$ on $U_0$. Since this vanishes on $\gamma_0$, we can extend $\Phi_0$
to the double $\widehat U_0$ of $U_0$ by the reflection with respect to $\gamma_0$.
The resulting function $\widehat \Phi_0$ is harmonic on $\widehat U_0$. In particular, $|\nabla \widehat \Phi_0|$ is continuous 
on $\widehat U_0$ and we may assume that this is bounded by $M$.
Then, the $p$-Dirichlet integral of $\Phi_0$ on $U_0$ is bounded by
$$
\int_{U_0}(\vert \nabla \Phi_0(z) \vert \rho_{R_0}^{-1}(z))^p\rho_{R_0}^2(z)dxdy 
\leq M^p \int_{U_0}\rho_{R_0}^{2-p}(z)dxdy \asymp \int_{r<|z|<1}\frac{dxdy}{(1-|z|^2)^{2-p}}<\infty.
$$
Since $\Psi$ has finite $p$-Dirichlet integral on $U_0$, so does $\Phi$.

Next, we prove that $P_{\partial R_0}$ is a bounded operator. Suppose to the contrary that $P_{\partial R_0}$ is not bounded.
Then, there exists a sequence $\psi_n \in B_p(\mathbb S)$ with $\Vert \psi_n \Vert_{B_p}=1$ such that
$\Vert \Phi_n \Vert_{D_p} \to \infty$ $(n \to \infty)$ for $\Phi_n=P_{\partial R_0}(\psi_n)$.
For $\Psi_n=\sigma^*(P_{\partial R_0}(\psi_n)|_{V_0})$, we also have
$$
\Vert \Psi_n|_{U_0} \Vert_{D_p} \leq \Vert P_{\partial R_0}(\psi_n) \Vert_{D_p} \asymp \Vert \psi_n \Vert_{B_p}=1.
$$
Since $B_p(\mathbb S) \subset {\rm BMO}(\mathbb S) \subset L_1(\mathbb S)/\mathbb C$ and the inclusion is continuous,
by adding constants to $\psi_n$, we may assume that $\Vert \psi_n \Vert_{L_1}$ are all bounded by some positive constant.
Then, $\{\Phi_n\}$ constitutes a normal family. Moreover, $\{\widehat \Phi_{0n}\}$ constructed
from $\Phi_n$ and $\Psi_n$ in the above argument also constitutes a normal family in $\widehat U_0$. 
By passing to a subsequence, we may assume that $\Phi_n$ converge uniformly to $\Phi_\infty$ in $R_0 \setminus U_0$,
and $\widehat \Phi_{0n}$ converge uniformly to $\widehat \Phi_{0\infty}$ in $\widehat U_0$. It follows that
$$
\limsup_{n \to \infty} \Vert \Phi_n \Vert_{D_p} \leq \Vert \Phi_\infty|_{R_0 \setminus U_0} \Vert_{D_p}+\Vert \Phi_{0\infty}|_{U_0} \Vert_{D_p}+
\limsup_{n \to \infty} \Vert \Psi_n|_{U_0} \Vert_{D_p}.
$$
However, the right side of the above inequality is bounded, which contradicts the assumption.
Thus, the proof of the boundedness is completed.
\end{proof}

\begin{proof}[Proof of Theorem \ref{extension}]
For any $\psi \in B_p(\mathbb S)$, let $\widetilde \Psi=P(\psi) \in {\rm HD}_p(\mathbb D)$ and
$\Psi=\sigma^*(\widetilde \Psi|_{V_0})$. 
We also set $\Phi=P_{\partial R_0}(\sigma^* \psi)=P_{\partial R_0} \circ E_{\partial R_0}(\Psi)$.
Since $\Phi$ has finite $p$-Dirichlet integral on $U_0$ and $\Phi$ and $\Psi$ have ``the same boundary values on $\gamma_0$'',
Lemma \ref{new} implies that $\Phi \in {\rm HD}_p(R_0)$. This shows that $E_{\partial R_0}$ is surjective and it is also injective.

The boundedness of $E_{\partial R_0}:\Phi \mapsto \psi$ can be verified by considering its inverse operator $P_{\partial R_0}$.
By Lemma \ref{new}, $P_{\partial R_0}$ is a bounded linear operator.
Then, the open mapping theorem implies that both are Banach isomorphisms.
\end{proof}

By the fact that $P_{\partial R_0}:B_p(\mathbb S) \to {\rm HD}_p(R_0)$ is a Banach isomorphism,
we see that the $p$-Dirichlet principle \eqref{r-dirichlet} is valid for a compact bordered Riemann surface.
We state it as a by-product of our arguments due to its own interest. It is also 
expected that this should be true for any Riemann surface $R$.

\begin{theorem}
For a compact bordered Riemann surface $R_0$, $D_p(R_0)$ admits the topological direct sum decomposition
\begin{equation}
D_p(R_0)={\rm HD}_p(R_0) \oplus \overline{C_0^\infty(R_0)}
\end{equation}
for any $p>1$.
\end{theorem}

\begin{proof}
We extend $E_{\partial R_0}$ to $D_p(R_0)$ as a bounded linear operator. 
Then, it follows that $D_p(R_0)={\rm Im}\,P_{\partial R_0} \oplus {\rm Ker}\,E_{\partial R_0}$,
which yields the desired decomposition.

For any $\Phi \in D_p(R_0)$,
we see that $\sigma_*(\Phi|_{U_0})$ is a function on $V_0$
whose $p$-Dirichlet norm with respect to $\rho_{\mathbb D}$ is finite by the comparison of
hyperbolic densities of $R_0$ and $\mathbb D$ under $\sigma$. We extend $\sigma_*(\Phi|_{U_0})$ to  
$\widetilde \Phi \in D_p(\mathbb D)$ so that their $p$-Dirichlet norms are comparable. Then,
Theorem \ref{dirichletp} yields $\widetilde \Psi \in {\rm HD}_p(\mathbb D)$
that has the same boundary extension to $\mathbb S$ as $\widetilde \Phi$. This boundary function $\psi$ on $\mathbb S$ is 
defined as $E_{\partial R_0}(\Phi)$.
By
$$
\Vert \psi \Vert_{B_p} \asymp \Vert \widetilde \Psi \Vert_{D_p} \lesssim \Vert \widetilde \Phi \Vert_{D_p} 
\lesssim \Vert \Phi \Vert_{D_p},
$$
we have the result. 
\end{proof}

The transmission operator $\Theta_\gamma:{\rm HD}_p(R_1) \to {\rm HD}_p(R_2)$ can be described
as in the planar case
once we fix the quasicircle $\Gamma=\sigma(\gamma)$.
For $\Psi_i \in {\rm HD}_p(R_i)$ $(i=1,2)$, we 
regard its boundary function on $\partial R_i=\gamma$ as a function on $\Gamma$ through $\sigma$.
Namely, it is the image of 
the boundary extension map $E_{\partial R_i}:{\rm HD}_p(R_i) \to B_p(\mathbb S)$ $(i=1,2)$ in Theorem \ref{extension}.
Then, the transmission is given
in terms of the composition operator $C_{h_\Gamma}$ of $B_p(\mathbb S)$ as follows.

\begin{lemma}\label{step}
The transmission operator $\Theta_\gamma:{\rm HD}_p(R_1) \to {\rm HD}_p(R_2)$
is represented as $P_{\partial R_2} \circ C_{h_\Gamma} \circ E_{\partial R_1}$,
where $h_\Gamma$ is the quasisymmetric homeomorphism of $\mathbb S$ associated with $\Gamma=\sigma(\gamma)$.
\end{lemma}
\smallskip

The boundedness of $\Theta_\gamma$ is proved in \cite[Theorem 3.29]{SS20} and \cite[Theorem 3.44]{SS21} for $p=2$.
Moreover, in account of Corollary \ref{main1}, we also see that
the boundedness of $\Theta_\gamma$ conversely characterize a quasicircle.
Thus, we have the following conclusion, 
which is proved in \cite[Theorem 1.1]{Shiga} for $p=2$ and for a general Riemann surface $R$.

\begin{theorem}\label{shiga}
Let $\gamma$ be a closed Jordan curve in a compact Riemann surface $R$.
If $\gamma$ is a quasicircle, then 
the transmission operator $\Theta_\gamma:{\rm HD}_p(R_1) \to {\rm HD}_p(R_2)$ is a Banach isomorphism.
Conversely, if $\Theta_\gamma$ gives a Banach isomorphism between ${\rm HD}_p(R_1)$ and
${\rm HD}_p(R_2)$, then $\gamma$ is a quasicircle.
\end{theorem}

\section{Variation of the space of Dirichlet finite harmonic functions}

In this section, we deform a Riemann surface $R$ quasiconformally and
examine how the spaces of Dirichlet finite harmonic functions are affected by this deformation.
In particular, we are interested in the case where $R$ is divided by a Weil--Petersson curve $\gamma$,
and consider the variation of the transmission of Dirichlet finite harmonic functions across $\gamma$.

\subsection{Quasiconformal deformation of Riemann surfaces}
Let $M(R)$ be the space of all Beltrami coefficients (differentials) $\mu=\mu(z)d\bar z/dz$ on $R$.
For $p>1$, we define
$M_p(R)$ as the subset of Beltrami coefficients
that are $p$-integrable with respect to the hyperbolic metric $\rho_{R}(z)|dz|$ on $R$;
$$
\Vert \nu \Vert_p=\left(\int_{R} |\nu(z)|^p\rho_{R}^2(z)dxdy \right)^{1/p}<\infty.
$$
The topology on $M_p(R)$ is defined by the norm $\Vert \nu \Vert_p+\Vert \nu \Vert_\infty$.
For any $\nu \in M(R)$, the quasiconformal homeomorphism of $R$ with complex dilatation $\nu$ is denoted by $g^\nu$ and
its image by $R^\nu=g^\nu(R)$. 

Let $R_0$ be a compact bordered Riemann surface whose boundary $\partial R_0$ is a simple closed curve $\gamma$.
We fix a conformal homeomorphism of a neighborhood $U_0$ of $\gamma$ in $R_0$ to a neighborhood $V_0$ of $\mathbb S$ in $\mathbb D$.
For any $\nu \in M(R_0)$, we take the Beltrami coefficient $\sigma_*(\nu|_{U_0})$ on $V_0$ by 
pushing $\nu|_{U_0}$ forward,
and extend it to $\mathbb D$ by $0$ on $\mathbb D \setminus V_0$; the resulting one is denoted by $\tilde \nu$.
We define this correspondence $\nu \in M(R_0) \mapsto \tilde \nu \in M(\mathbb D)$ 
as a map $Q:M(R_0) \to M(\mathbb D)$. Obviously, $Q$ is continuous in the topology of $\Vert \cdot \Vert_\infty$, but we see more.

\begin{lemma}\label{L}
For any $\nu \in M_p(R_0)$, $Q(\nu)=\tilde \nu$ belongs to $M_p(\mathbb D)$. Moreover,
$Q:M_p(R_0) \to M_p(\mathbb D)$ is continuous.
\end{lemma}

\begin{proof}
By the comparison of hyperbolic densities,
we have
\begin{equation*}
\int_{V_0}|\tilde \nu(z)|^p\rho_{\mathbb D}^2(z)dxdy  \asymp
\int_{U_0}|\nu(z)|^p\rho_{{R_0}}^2(z)dxdy \leq \int_{{R_0}}|\nu(z)|^p\rho_{{R_0}}^2(z)dxdy.
\end{equation*}
Hence, $\tilde \nu$ belongs to $M_p(\mathbb D)$. 
Suppose that a sequence $\nu_n$ converges to $\nu$ in $M_p({R_0})$. 
This means that $\Vert \nu-\nu_n \Vert_p \to 0$ and
$\Vert \nu-\nu_n \Vert_\infty \to 0$ as $n \to \infty$. Then, the continuity of
$Q$ on $M(R_0)$ implies that $\Vert \tilde \nu-\tilde \nu_n \Vert_\infty \to 0$.
For the convergence on the $L_p$-norm, the above estimate can be applied.
This proves that $Q$ is continuous.
\end{proof}

We deform a compact bordered Riemann surface with one boundary component 
quasiconformally and observe its effect on the space of Dirichlet finite harmonic functions.
Let $R_0$ be a compact bordered Riemann surface whose boundary $\partial R_0$ is a simple closed curve $\gamma$.
For any $\nu \in M_p({R_0})$, we consider the quasiconformal deformation $g^{\nu}:R_0 \to R_0^\nu$.
Let $\tilde \nu=Q(\nu) \in M_p(\mathbb D)$.
The normalized quasiconformal self-homeomorphism $H^{\tilde \nu}$ of $\mathbb D$
whose complex dilatation is $\tilde \nu$ induces a quasisymmetric homeomorphism
$h^{\tilde \nu}:\mathbb S \to \mathbb S$ by its extension to $\mathbb S$.

In Theorem \ref{extension}, we have seen that 
the boundary extension map $E_{\partial R_0}:{\rm HD}_p({R_0}) \to B_p(\mathbb S)$ 
is a Banach isomorphism. Hence, its quasiconformal deformation is also a Banach isomorphism.

\begin{proposition}\label{qc} 
For any $\nu \in M_p(R_0)$, let $g^{\nu}:R_0 \to R_0^\nu$ be the quasiconformal homeo\-morphism and
$h^{\tilde \nu}:\mathbb S \to \mathbb S$ the associated quasisymmetric homeomorphism.
Then, 
$$
P_{\partial R_0^\nu} \circ (h^{\tilde \nu})_* \circ E_{\partial {R_0}}:{\rm HD}_p(R_0) \to {\rm HD}_p(R_0^\nu)
$$ 
is a Banach isomorphism. 
\end{proposition}

\subsection{Deformation of Weil--Petersson curves}
Just like a quasicircle on a Riemann surface $R$, we define a $p$-Weil--Petersson curve on $R$ and
consider its deformation induced by certain quasiconformal homeomorphisms of $R$.

\begin{definition}
We say that a closed Jordan curve $\gamma$ on a Riemann surface $R$ is
a {\em $p$-Weil--Petersson curve} if there exist an annular neighborhood $U$ of $\gamma$ in
$R$ and a conformal embedding $\sigma:U \to \widehat{\mathbb C}$ such that the image $\sigma(\gamma)$ is 
a $p$-Weil-Petersson curve in $\widehat{\mathbb C}$.
\end{definition}
\smallskip

This is well defined independently of the choice of $U$ and $\sigma$.
We can verify this as follows.
Suppose that there are another neighborhood $U'$ and another conformal embedding $\sigma':U' \to \widehat{\mathbb C}$.
If $\sigma(\gamma)$ is a $p$-Weil-Petersson curve, then
$\sigma(\gamma)=G(\mu_1,\mu_2)(\mathbb S)$ for some $\mu_1 \in M_p(\mathbb D)$ and $\mu_2 \in M_p(\mathbb D^*)$.
We define $\widehat G=\sigma' \circ \sigma^{-1}\circ G(\mu_1,\mu_2)$ 
on the neighborhood $G(\mu_1,\mu_2)^{-1}(\sigma(U \cap U'))$ of $\mathbb S$.
The complex dilatation of $\widehat G$ on this neighborhood of $\mathbb S$ is just the restriction of
$\mu_1$ and $\mu_2$. We extend $\widehat G$ quasiconformally to $\widehat {\mathbb C}$ in a suitable way
by taking a smaller neighborhood if necessary (see \cite[Chap.II, Theorem 8.1]{LV}).
We denote the resulting quasiconformal self-homeomorphism of $\widehat {\mathbb C}$ 
by $G(\mu_1',\mu_2')$. Since $\mu_1=\mu_1'$ and $\mu_2=\mu_2'$ on the neighborhood of $\mathbb S$,
we see that $\mu_1' \in M_p(\mathbb D)$ and $\mu_2' \in M_p(\mathbb D^*)$. Moreover,
$G(\mu_1',\mu_2')(\mathbb S)=\sigma'(\gamma)$.
Hence, $\sigma'(\gamma)$ is a $p$-Weil--Petersson curve by Proposition \ref{double}.

As before, let $R$ be a compact Riemann surface and $\gamma$ a $p$-Weil--Petersson curve in $R$
that divides $R$ into $R_1$ and $R_2$.
Suppose that there are an annular neighborhood $U$ and a conformal embedding
$\sigma:U \to \widehat{\mathbb C}$ such that $\sigma(\gamma)=\mathbb S$. 
This assumption is not essential and imposed for the sake of simplicity of notation.
We may start with the usual setting where $\sigma(\gamma)$ is merely a $p$-Weil--Petersson curve.

Let $M_p(R_i)$ be the space of $p$-integrable Beltrami differentials on $R_i$ and
$T_p(R_i)$ the $p$-integrable Teich\-m\"ul\-ler space for $R_i$ given by the set of Teichm\"uller
equivalence classes of $M_p(R_i)$ $(i=1,2)$. This integrable Teich\-m\"ul\-ler space for 
a Riemann surface in general is introduced in \cite{Yan} for $p \geq 2$, and also in \cite{RSS}
for a compact bordered Riemann surface. 
The extension to the case $p>1$ is also possible; at least the metric space $T_p(R)$ can be constructed in the same way.
The ordinary Teichm\"uller space $T(R)$ for $R$ is the set of Teichm\"uller
equivalence classes of $M(R)$.

Let $g(\nu_1, \nu_2):R \to R(\nu_1, \nu_2)$ be the quasiconformal deformation of $R$
determined by $([\nu_1],[\nu_2]) \in T_p(R_1) \times T_p(R_2)$.
Here, $R(\nu_1, \nu_2)$ is the Riemann surface determined by this deformation.
We call $g(\nu_1, \nu_2)$ 
a $p$-Weil--Petersson quasiconformal homeomorphism of $R$ with respect to $\gamma$.
Then, a closed Jordan curve $g(\nu_1, \nu_2)(\gamma)$ in $R(\nu_1, \nu_2)$ is well defined by $([\nu_1],[\nu_2])$,
and it divides $R(\nu_1, \nu_2)$ into $R_1^{\nu_1}$ and $R_2^{\nu_2}$.
By Proposition \ref{double}, we see that $g(\nu_1, \nu_2)(\gamma)$ is also a $p$-Weil--Petersson curve on $R(\nu_1, \nu_2)$.

\subsection{Strong convergence of the transmission operators}

We continue the setting in the previous subsection, where a compact Riemann surface $R$
is divided into $R_1$ and $R_2$ by a $p$-Weil--Petersson curve $\gamma$ such that a conformal 
embedding $\sigma:U \to \widehat{\mathbb C}$ maps $\gamma$ onto $\mathbb S$.
Under the $p$-Weil--Petersson quasiconformal deformation $g(\nu_1, \nu_2):R \to R(\nu_1, \nu_2)$
defined by $([\nu_1],[\nu_2]) \in T_p(R_1) \times T_p(R_2)$,
we investigate the transmission operator 
$$
\Theta_{g(\nu_1, \nu_2)(\gamma)}:{\rm HD}_p(R_1^{\nu_1}) \to {\rm HD}_p(R_2^{\nu_2}).
$$

By restricting $\nu_1$ and $\nu_2$ to $U$, pushing them forward to $V$ by $\sigma$, 
and extending to the outside of $V$ by $0$,
we obtain $\tilde \nu_1 \in M_p(\mathbb D)$ and $\tilde \nu_2 \in M_p(\mathbb D^*)$.
These Beltrami coefficients 
define the pair $([\tilde \nu_1],[\tilde \nu_2])$ in $T_p(\mathbb D) \times T_p(\mathbb D^*)$.
By Lemma \ref{L},
this correspondence 
$$
\widetilde Q:T_p(R_1) \times T_p(R_2) \to T_p(\mathbb D) \times T_p(\mathbb D^*)
$$
is continuous. 

We vary $([\nu_1],[\nu_2])$ in $T_p(R_1) \times T_p(R_2)$,
and consider the convergence of a sequence of the transmission operators $\Theta_{g(\nu_1,\nu_2)(\gamma)}$.
We set up the situation in the following two claims.
The first one comes from the planar case.

\begin{proposition}\label{compd}
For $\mu_1 \in M_p(\mathbb D)$ and $\mu_2 \in M_p(\mathbb D^*)$, 
the quasisymmetric homeomorphism $h_\Gamma$ of $\mathbb S$ associated with 
the $p$-Weil--Petersson curve $\Gamma=G(\mu_1,\mu_2)(\mathbb S)$
is given as $h_\Gamma=h^{\mu_1} \circ h^{\bar \mu_2^{-1}}$.
\end{proposition}

\begin{proof}
By Proposition \ref{double}, we have 
$\Gamma=F_{\mu_1 \ast \bar \mu_2^{-1}}(\mathbb S)=F_{\mu_1} \circ H^{\bar \mu_2^{-1}}(\mathbb S)$.
From this, the statement follows.
\end{proof}

\begin{lemma}\label{setup}
For $\nu_1 \in M_p(R_1)$ and $\nu_2 \in M_p(R_2)$, let $\Gamma=G(\tilde \nu_1,\tilde \nu_2)(\mathbb S)$
be the $p$-Weil--Petersson curve defined by $([\tilde \nu_1],[\tilde \nu_2])=\widetilde Q([\nu_1],[\nu_2])$.
Then, the transmission operator $\Theta_{g(\nu_1,\nu_2)(\gamma)}$ is represented as
$$
\Theta_{g(\nu_1,\nu_2)(\gamma)}
=P_{\partial R_2^{\nu_2}} \circ (h^{{\bar{\tilde \nu}}_2})_* \circ
(h^{\tilde \nu_1})_*^{-1} \circ E_{\partial R_1^{\nu_1}}
=P_{\partial R_2^{\nu_2}} \circ C_{h_\Gamma} \circ E_{\partial R_1^{\nu_1}}.
$$
\end{lemma}

\begin{proof}
Proposition \ref{qc} gives
\begin{align}
P_{\partial R_1^{\nu_1}} \circ (h^{\tilde \nu_1})_* \circ E_{\partial {R_1}}&:{\rm HD}_p(R_1) \to {\rm HD}_p(R_1^{\nu_1});\\
P_{\partial R_2^{\nu_2}} \circ (h^{{\bar{\tilde \nu}}_2})_* \circ E_{\partial {R_2}}&:{\rm HD}_p(R_2) \to {\rm HD}_p(R_2^{\nu_2}),
\end{align}
and Lemma \ref{step} in particular implies that
$\Theta_{\mathbb S}=P_{\partial R_2} \circ E_{\partial R_1}$.
Hence, by these representations, we have
$$
\Theta_{g(\nu_1,\nu_2)(\gamma)}=P_{\partial R_2^{\nu_2}} \circ (h^{{\bar{\tilde \nu}}_2})_* \circ
(h^{\tilde \nu_1})_*^{-1} \circ E_{\partial R_1^{\nu_1}}.
$$
Since 
$C_{h_\Gamma}=(h^{\tilde \nu_1} \circ h^{{\bar{\tilde \nu}}_2^{-1}})^*=(h^{{\bar{\tilde \nu}}_2})_* \circ
(h^{\tilde \nu_1})_*^{-1}$ 
by Proposition \ref{compd},
we have the representation as in the statement.
\end{proof}

The main result can be formulated as follows. 
Under the representation of the transmission operators as in Lemma \ref{setup}, this is immediately verified by 
Proposition \ref{strong} with
the continuity of $\widetilde Q$. 

\begin{theorem}
Let $R$ be a compact Riemann surface and $\gamma$ a $p$-Weil--Petersson curve
which divides $R$ into $R_1$ and $R_2$.
Suppose that a sequence of pairs $([(\nu_1)_n],[(\nu_2)_n])$ converges to $([(\nu_1)_0],[(\nu_2)_0])$
in $T_p(R_1) \times T_p(R_2)$. 
Then, the transmission operators $\Theta_{\gamma_n}:{\rm HD}_p(R_1^{(\nu_1)_n}) \to {\rm HD}_p(R_2^{(\nu_2)_n})$ 
for the Weil--Petersson curves $\gamma_n=g((\nu_1)_n, (\nu_2)_n)(\gamma)$ 
on the Riemann surfaces $R((\nu_1)_n, (\nu_2)_n)$ $(n \geq 0)$ converge strongly to
$\Theta_{\gamma_0}$ in the sense that the corresponding composition operators $C_{h_{\Gamma_n}}$ acting on $B_p(\mathbb S)$ converge strongly to $C_{h_{\Gamma_0}}$,
where $\Gamma_n$ are the Weil-Petersson curves $G((\tilde \nu_1)_n,(\tilde \nu_2)_n)(\mathbb S)$ 
in $\widehat{\mathbb C}$
for $([(\tilde \nu_1)_n],[(\tilde \nu_2)_n])=\widetilde Q([(\nu_1)_n],[(\nu_2)_n])$.
\end{theorem}
\smallskip

In the case $p=2$, by using Theorem \ref{strong0} instead of Proposition \ref{strong}, we obtain the stronger assertion
than the above theorem for a quasicircle $\gamma$ and quasiconformal deformation of $R$.

\begin{theorem}
Let $R$ be a compact Riemann surface and $\gamma$ a quasicircle
which divides $R$ into $R_1$ and $R_2$.
Suppose that a sequence $[\nu_n]$ converges to $[\nu_0]$
in $T(R)$. 
Then, the transmission operators $\Theta_{\gamma_n}:{\rm HD}_2(R_1^{\nu_n}) \to {\rm HD}_2(R_2^{\nu_n})$ 
for the quasicircles $\gamma_n=g^{\nu_n}(\gamma)$ 
on the Riemann surfaces $R^{\nu_n}$ $(n \geq 0)$ converge strongly to
$\Theta_{\gamma_0}$ in the sense that the corresponding composition operators $C_{h_{\Gamma_n}}$ acting on $B_2(\mathbb S)$ converge strongly to $C_{h_{\Gamma_0}}$,
where $\Gamma_n$ are the quasicircles given as the image of $\mathbb S$ under the quasiconformal self-homeomorphism of $\widehat{\mathbb C}$ whose complex dilatation is $\sigma_*(\nu_n|_U)$ on $V$. 
\end{theorem}
\smallskip

In this case, the convergence $\lim_{n \to \infty} \Vert \Theta_{\gamma_n} \Vert=\Vert \Theta_{\gamma_0} \Vert$ is
shown in \cite[Theorem 5.7]{Shiga}.


\begin{thebibliography}{99}

\bibitem{Ah} L. V. Ahlfors, {\it Conformal Invariants: Topics in Geometric Function Theory}, AMS Chelsea Publishing, (2010).

\bibitem{Bi} C. J. Bishop, ``Function theoretic characterizations of Weil--Petersson curves,'' 
{\it Rev. Mat. Iberoam.}, {\bf 38}: 2355--2384 (2022).

\bibitem{Bi1} C. J. Bishop, 
``Weil--Petersson curves, $\beta$-numbers, and minimal surfaces,'' {\it Ann. of Math.} 
{\bf 202}: 111--188 (2025).

\bibitem{B2} G. Bourdaud, ``Changes of variable in Besov spaces II,'' {\it Forum Math.} {\bf 12}: 545--563 (2000).

\bibitem{BS} G. Bourdaud and W. Sickel, ``Changes of variable in Besov spaces,'' {\it Math. Nachr.} {\bf 198}: 19--39 (1999).

\bibitem{Cu} G. Cui, Integrably asymptotic affine homeomorphisms of the circle and Teich\-m\"ul\-ler spaces, 
{\it Sci. China Ser. A} {\bf 43}: 267--279 (2000).

\bibitem{Gu} H. Guo, ``Integrable Teichm\"uller spaces,'' {\it Sci. China Ser. A} {\bf 43}: 47--58 (2000). 

\bibitem{JK} D. S. Jerison and C. E. Kenig, ``Hardy spaces, $A_\infty$, and singular integrals on chord-arc domains,''
{\it Math. Scand.} {\bf 50}: 221--247 (1982).

\bibitem{Le} O. Lehto, 
{\it Univalent Functions and Teich\-m\"ul\-ler Spaces}, 
Grad. Texts in Math. 109, Springer, (1987).

\bibitem{LV} O. Lehto and K. I. Virtanen, {\it Quasiconformal Mappings in the Plane}, 2nd ed., Grund. math. Wiss. 126, Springer, (1973). 

\bibitem{Mnew} K. Matsuzaki,
{\it Integrable Teichm\"uller spaces for analysis on Weil--Petersson curves,} arXiv.2508.20341.

\bibitem{Maz} V. Maz'ya, {\it Sobolev Spaces},
Grund. math. Wiss. 342, Springer, (2011).

\bibitem{MR} P. Mironescu and E. Russ,
``Traces of weighted Sobolev spaces. Old and new,'' {\it Nonlinear Analysis} {\bf 119}: 354--381 (2015).

\bibitem{NS} S. Nag and D. Sullivan, ``Teich\-m\"ul\-ler theory and the universal period mapping via quantum calculus and 
the $H^{1/2}$ space on the circle,'' {\it Osaka J. Math.} {\bf 32}: 1--34 (1995).

\bibitem{RSS} D. Radnell, E. Schippers, and W. Staubach, ``A Hilbert manifold structure on the Weil--Petersson class Teichm\"uller space of bordered Riemann surfaces,'' {\it Commun. Contemp. Math.} {\bf 17}: 1550016
 (2015). 

\bibitem{SN} L. Sario and M. Nakai, {\it Classification Theory of Riemann Surfaces}, Grund. math. Wiss. 164, Springer, (1970).

\bibitem{SS0} E. Schippers and W. Staubach, ``Harmonic reflection in quasicircles and well-posedness of a Riemann--Hilbert problem on quasidisks,'' {\it J. Math. Anal. Appl.} {\bf 448}: 864--884 (2017). 

\bibitem{SS} E. Schippers and W. Staubach, ``Transmission of harmonic functions through quasicircles on compact
Riemann surfaces,'' Ann. Acad. Sci. Fenn. Math. {\bf 45}: 1111--1134 (2020).

\bibitem{SS20} E. Schippers and W. Staubach, ``Plemelj--Sokhotski isomorphism for quasicircles in Riemann
surfaces and the Schiffer operators,'' Math. Ann. {\bf 378}: 1613--1653 (2020).

\bibitem{SS21} E. Schippers and W. Staubach, ``A Scattering theory of harmonic
one-forms on Riemann surfaces,'' arXiv.2112.00835 (2021).

\bibitem{SS1} E. Schippers and W. Staubach, ``Analysis on quasidisks: a unified approach through transmission and jump problems,''
{\it EMS Surv. Math. Sci} {\bf 9}: 31--97 (2022).

\bibitem{Sh} Y. Shen, ``Weil--Petersson Teich\-m\"ul\-ler space,'' 
{\it Amer. J. Math.} {\bf 140}: 1041--1074 (2018). 

\bibitem{Shiga} H. Shiga, ``Quasicircles and Dirichlet finite harmonic functions on Riemann surfaces,''
{\it Essays in geometry --\,dedicated to Norbert A'Campo}, IRMA Lect. Math. Theor. Phys. 34, EMS Press, 155--178 (2023).

\bibitem{St} E. Stein, {\it Singular Integrals and Differentiability Properties of Functions}, Princeton Univ. Press, (1970). 

\bibitem{TT} L. Takhtajan and L. P. Teo, ``Weil--Petersson metric on the universal Teich\-m\"ul\-ler space,'' {\it Mem. Amer. Math. Soc.} 
{\bf 183}: 861 (2006).

\bibitem{TS} S. Tang and Y. Shen, ``Integrable Teich\-m\"ul\-ler space,''
{\it J. Math. Anal. Appl.} {\bf 465}: 658--672 (2018). 

\bibitem{WM-0} H. Wei and K. Matsuzaki, ``The $p$-Weil--Petersson Teichm\"uller space and the quasiconformal extension of curves,''
J. Geom. Anal. {\bf 32}: 213 (2022).

\bibitem{WM-1} H. Wei and K. Matsuzaki, ``The $p$-integrable Teich\-m\"ul\-ler space for $p \geqslant 1$,'' 
Proc. Japan Acad. Ser. A Math. Sci. {\bf 99}: 37--42 (2023), arXiv:2210.04720.

\bibitem{WM-4} H. Wei and K. Matsuzaki, ``Parametrization of the $p$-Weil--Petersson curves: holomorphic dependence,''
J. Geom. Anal. {\bf 33}: 292 (2023).

\bibitem{WZ} H. Wei and M. Zinsmeister,
``Dirichlet spaces over chord-arc domains,'' {\it Math. Ann.} {\bf 391}: 1045--1064 (2025). 

\bibitem{WZp} H. Wei and M. Zinsmeister,
``$p$-Dirichlet spaces over chord-arc domains,'' arXiv:2410.02183.

\bibitem{Yan} M. Yanagishita, ``Introduction of a complex structure on the $p$-integrable Teichm\"uller space,'' 
{\it Ann. Acad. Sci. Fenn. Math.} {\bf 39}: 947--971 (2014).

\bibitem{Zhu} K. Zhu, {\it Operator Theory in Function Spaces}, 
Math. Surveys Mono. 138, American Math. Soc. (2007).

\end{thebibliography}
\end{document}